\documentclass[12pt]{amsart}
\usepackage[margin=1.2in]{geometry}
\usepackage{graphicx,latexsym}
\usepackage{comment}
\usepackage{mathrsfs}
\usepackage{enumitem}
\usepackage{tikz}
\usepackage{amsfonts, amssymb, amsmath, amsthm, bm, hyperref}
\usepackage[utf8]{inputenc}
\hypersetup{
  colorlinks=true,
}
\usepackage[all]{xy}
\xyoption{ps}
\usepackage[backend=biber, style=numeric-comp, natbib=false, giveninits=true, url=false, isbn=false, doi=true, eprint=true]{biblatex}

\renewbibmacro{in:}{%
\ifentrytype{article}{}{\printtext{\bibstring{in}\intitlepunct}}}

\newtheorem{theorem}{Theorem}[section]
\newtheorem{proposition}[theorem]{Proposition}

\newtheorem{lemma}[theorem]{Lemma}
\newtheorem{corollary}[theorem]{Corollary}

\theoremstyle{definition}

\theoremstyle{remark}

\newcommand{\N}{\mathbb{N}}
\newcommand{\Z}{\mathbb{Z}}
\newcommand{\R}{\mathbb{R}}
\newcommand{\C}{\mathbb{C}}
\newcommand{\cD}{\mathscr{D}}
\newcommand{\cS}{\mathscr{S}}
\newcommand{\cO}{\mathscr{O}}
\newcommand{\cE}{\mathscr{E}}
\newcommand{\cB}{\mathscr{B}}

\newcommand{\abso}[1]{{\lvert#1\rvert}}

\newcommand{\dx}{{\rm d}x }

\newcommand{\dt}{{\rm d}t }

\numberwithin{equation}{section}

\begin{document}

\title[Sequence space representations via Wilson bases]
{Sequence space representations for spaces of smooth functions and distributions via Wilson bases}

\author{C.~Bargetz}
\address{C.~Bargetz, Department of Mathematics\\Universität Innsbruck\\Technikerstraße 13\\6020 Innsbruck\\Austria}
\email{christian.bargetz@uibk.ac.at}

\author{A.~Debrouwere}
\address{A.~Debrouwere, Department of Mathematics and Data Science \\ Vrije Universiteit Brussel, Belgium\\ Pleinlaan 2 \\ 1050 Brussels \\ Belgium}
\email{andreas.debrouwere@vub.be}

\author{E.~A.~Nigsch}
\address{E.~A.~Nigsch, Institute of Analysis and Scientific Computing, TU Wien, 1040 Vienna, Austria}
\email{eduard.nigsch@tuwien.ac.at}

\begin{abstract} 
We provide explicit sequence space representations for  the test function and distribution spaces occurring in the Valdivia-Vogt structure tables by making use of Wilson bases generated by compactly supported smooth windows. Furthermore, we show that these kind of bases are common unconditional Schauder bases for all separable spaces occurring in these tables. Our work implies that the Valdivia-Vogt structure tables for test functions and distributions may  be interpreted as one large commutative diagram. 
\end{abstract}
\maketitle
\section{Introduction}
Already in~\cite[p.~262]{Schwartz} L.~Schwartz observed that the Hermite functions~$(H_n)_{n=1}^{\infty}$  form a basis of $\cS$ and that the mapping
\begin{equation} 
 \cS \to s, \, f\mapsto (\langle f, H_n\rangle)_{n=1}^{\infty},
 \label{hermite}
\end{equation}
is an isomorphism of locally convex spaces. Around 1980, M.~Valdivia and D.~Vogt~\cite{ValdiviaDK, Valdivia, Val80:CertainSpaces, Val81:DLp, Val79:RepresenationD, Val82:LCS, Vogt} found sequence space representations for many spaces of (generalized) functions. We refer to \cite{B14,B15,BargetzOC, BExplicit, Debrouwere, C-G-P-R, Langenbruch2012, Langenbruch2016} for more recent works on this topic. N.~Ortner and P.~Wagner presented in~\cite{OWDLp} the known sequence space representations for the spaces occurring in  L.~Schwartz' theory of distributions \cite{Schwartz} in two tables which, in updated form, are

\begin{equation}\label{table_functions}
  \begin{gathered}
    \xymatrix@R=1.5em@C=0.8em{
      \cD \ar@{}[r]|-*[@]\txt{$\subseteq$} \ar@{}[d]|-*\txt{\rotatebox[origin=c]{-90}{$\cong$}} &
      \cD^F \ar@{}[r]|-*[@]\txt{$\subseteq$} \ar@{}[d]|-*\txt{\rotatebox[origin=c]{-90}{$\cong$}} &
      \cS \ar@{}[r]|-*[@]\txt{$\subseteq$} \ar@{}[d]|-*\txt{\rotatebox[origin=c]{-90}{$\cong$}} &
      \cD_{L^p} \ar@{}[r]|-*[@]\txt{$\subseteq$} \ar@{}[d]|-*\txt{\rotatebox[origin=c]{-90}{$\cong$}} &
      \dot\cB \ar@{}[r]|-*[@]\txt{$\subseteq$} \ar@{}[d]|-*\txt{\rotatebox[origin=c]{-90}{$\cong$}} &
      \cD_{L^\infty} \ar@{}[r]|-*[@]\txt{$\subseteq$} \ar@{}[d]|-*\txt{\rotatebox[origin=c]{-90}{$\cong$}} &
      \cO_C \ar@{}[r]|-*[@]\txt{$\subseteq$} \ar@{}[d]|-*\txt{\rotatebox[origin=c]{-90}{$\cong$}} &
      \cO_M \ar@{}[r]|-*[@]\txt{$\subseteq$} \ar@{}[d]|-*\txt{\rotatebox[origin=c]{-90}{$\cong$}} &
      \cE \ar@{}[d]|-*\txt{\rotatebox[origin=c]{-90}{$\cong$}}
      \\
      \C^{(\N)} \widehat\otimes_\iota s \ar@{}[r]|-*[@]\txt{$\subseteq$} &
      \C^{(\N)} \widehat\otimes s \ar@{}[r]|-*[@]\txt{$\subseteq$} &
      s \widehat\otimes s \ar@{}[r]|-*[@]\txt{$\subseteq$} &
      \ell^p \widehat\otimes s \ar@{}[r]|-*[@]\txt{$\subseteq$} &
      c_0 \widehat\otimes s \ar@{}[r]|-*[@]\txt{$\subseteq$} &
      \ell^\infty \widehat\otimes s \ar@{}[r]|-*[@]\txt{$\subseteq$} &
      s' \widehat\otimes_\iota s \ar@{}[r]|-*[@]\txt{$\subseteq$} &
      s' \widehat\otimes s \ar@{}[r]|-*[@]\txt{$\subseteq$} &
      \C^\N \widehat\otimes s
    }
  \end{gathered}
\end{equation}

\begin{equation}\label{table_distributions}
  \begin{gathered}
    \xymatrix@R=1.5em@C=1em{
      \cE' \ar@{}[r]|-*[@]\txt{$\subseteq$} \ar@{}[d]|-*\txt{\rotatebox[origin=c]{-90}{$\cong$}} &
      \cO_M' \ar@{}[r]|-*[@]\txt{$\subseteq$} \ar@{}[d]|-*\txt{\rotatebox[origin=c]{-90}{$\cong$}} &
      \cO_C' \ar@{}[r]|-*[@]\txt{$\subseteq$} \ar@{}[d]|-*\txt{\rotatebox[origin=c]{-90}{$\cong$}} &
      \cD'_{L^p} \ar@{}[r]|-*[@]\txt{$\subseteq$} \ar@{}[d]|-*\txt{\rotatebox[origin=c]{-90}{$\cong$}} &
      \dot\cB' \ar@{}[r]|-*[@]\txt{$\subseteq$} \ar@{}[d]|-*\txt{\rotatebox[origin=c]{-90}{$\cong$}} &
      \cD'_{L^\infty} \ar@{}[r]|-*[@]\txt{$\subseteq$} \ar@{}[d]|-*\txt{\rotatebox[origin=c]{-90}{$\cong$}} &
      \cS' \ar@{}[r]|-*[@]\txt{$\subseteq$} \ar@{}[d]|-*\txt{\rotatebox[origin=c]{-90}{$\cong$}} &
      \cD' \ar@{}[d]|-*\txt{\rotatebox[origin=c]{-90}{$\cong$}}
      \\
      \C^{(\N)} \widehat\otimes s' \ar@{}[r]|-*[@]\txt{$\subseteq$} &
      s \widehat\otimes_\iota s' \ar@{}[r]|-*[@]\txt{$\subseteq$} &
      s \widehat\otimes s' \ar@{}[r]|-*[@]\txt{$\subseteq$} &
      \ell^p \widehat\otimes s'  \ar@{}[r]|-*[@]\txt{$\subseteq$} &
      c_0 \widehat\otimes s'  \ar@{}[r]|-*[@]\txt{$\subseteq$} &
      \ell^\infty \widehat\otimes s' \ar@{}[r]|-*[@]\txt{$\subseteq$} &
      s' \widehat\otimes s' \ar@{}[r]|-*[@]\txt{$\subseteq$} &
      \C^{\N} \widehat\otimes s' .
    }
  \end{gathered}
\end{equation}
The tables \eqref{table_functions} and \eqref{table_distributions} are often called the \emph{Valdivia-Vogt structure tables}. We refer to the book of L.~Schwartz \cite{Schwartz} (see also \cite{Horvath}) for the definitions of the test function and distribution spaces occurring in \eqref{table_functions} and \eqref{table_distributions}. The sequence spaces occurring in these tables will be discussed in Section \ref{sect-sequence spaces} below.

In contrast to the isomorphism $\cS \cong s$ induced by the mapping \eqref{hermite}, the original proofs of most of the isomorphisms in \eqref{table_functions} and \eqref{table_distributions} were obtained in a non-constructive fashion as they were based on the Pe{\l}czy\'{n}ski decomposition method.  Hence, it becomes an interesting problem to construct explicit sequence space representations for the spaces occurring in these tables. This question was studied e.g.\ in \cite{B14, BExplicit, OWDLp}. In the present paper, we provide  explicit mappings inducing the isomorphisms in \eqref{table_functions} and \eqref{table_distributions} by making use of \emph{Wilson bases}, a remarkable class of orthonormal bases of $L^2$ constructed by  I.~Daubechies  et.~al \cite{DJJ} in the context of time-frequency analysis \cite{FTFA}.  

In~\cite{B14} the first author showed that the tables \eqref{table_functions} and~\eqref{table_distributions}  may be interpreted as  commutative diagrams. More precisely, he constructed an isomorphism $\Psi \colon \cE \to \C^\N\widehat{\otimes} s$  such that all isomorphisms in~\eqref{table_functions}  can be chosen as restrictions of $\Psi$. By restricting and transposing, it was shown that a similar result holds for the table \eqref{table_distributions}. Since the spaces in~\eqref{table_functions} are embedded into $\cD'$, it seems natural to ask whether these  two tables may be interpreted as one large commutative diagram. Our main result implies that this is indeed the case. Namely, the following result holds, which was  one of the main motivations for this paper.

\begin{theorem}\label{thm-intro} The tables \eqref{table_functions} and~\eqref{table_distributions} may be combined into a single  commutative diagram, i.e., there exists an isomorphism $\Phi: \cD' \rightarrow \C^{\N} \widehat\otimes s'$ such
 that the restriction of $\Phi$  to  each of the spaces in the upper row of  \eqref{table_functions} and \eqref{table_distributions}  is an isomorphism onto the corresponding space in the lower row.
\end{theorem}

We now briefly explain Wilson bases and how they will be used in our work. We write 
\[
T_xf(t) = f(t-x), \qquad M_\xi f(t) = f(t)e^{2\pi i \xi t}, \qquad  x,\xi \in \R,
\]
for the translation and modulation operators on $\R$.  Given a window $\psi \in L^2(\R)$, the Wilson system  $\mathcal{W}(\psi)$  is defined as the following collection of linear combinations of time-frequency shifts of $\psi$:
\begin{equation}\label{eq:Wilson1D}
  \begin{aligned}
    \psi_{k,0} &= T_{k} \psi, \qquad k \in \Z, \\
    \psi_{k,n} &= \frac{1}{\sqrt{2}} T_{\frac{k}{2}}(M_n+ (-1)^{k+n} M_{-n})\psi, \qquad (k,n)\in\Z\times\N_{>0}.
  \end{aligned}
\end{equation}
Suppose that $\psi(x) = \overline{\psi(-x)}$. In \cite{DJJ} it is shown that $\mathcal{W}(\psi)$ is an orthonormal basis for $L^2(\R)$ (called a Wilson basis) if and only if 
\begin{equation}\label{eq:Wilson-cond}
 \sum_{k \in \Z} T_{n+ \frac{k}{2}}\psi T_{\frac{k}{2}} \overline{\psi} = 2\delta_{n,0}, \qquad n \in \N.
\end{equation}
 We refer to  \cite[Corollary 8.5.4]{FTFA} for various other characterizations of Wilson bases. Note that \eqref{eq:Wilson-cond} implies that there exists a Wilson basis $\mathcal{W}(\psi)$ generated by a window $\psi \in \cD(\R)$ (cf.\ \cite[Corollary 8.5.5(b)]{FTFA}); the existence of such a basis is rather surprising in view of the ``no-go" Balian-Low theorem \cite[Theorem 8.4.1]{FTFA} (see also \cite[Corollary 7.5.2]{FTFA}).
 
 As a rule of thumb, most global spaces of (generalized) functions may be characterized in terms  of the short-time Fourier transform; see e.g. \cite{B-O, DV21, G-Z, K-P-S-V}. In other words, they can be written as (topological) unions and/or intersections of the weighted modulation spaces $M^{p,q}_m$ \cite[Chapter 11]{FTFA}. Furthermore, Wilson bases provide sequence space representations for the spaces $M^{p,q}_m$ \cite[Theorem 12.3.4]{FTFA}. Inspired by these two observations, we show in this paper that 
Wilson bases may be used to obtain explicit sequence space representations for the  spaces occurring in \eqref{table_functions} and \eqref{table_distributions}. More precisely, we shall prove the following result (see also Theorem \ref{USSR} below).

\begin{theorem}\label{Wilson-intro}
 Let $\psi \in \cD(\R)$ be such that $\mathcal{W}(\psi)$ is a Wilson basis for $L^2(\R)$. Then,  the mapping
 $$
 \widetilde{C}_{\psi} : \cD'(\R) \rightarrow  \C^{\Z} \widehat\otimes s'(\N), \, f \mapsto (\langle f, \overline{\psi_{k,n}} \rangle )_{k,n}
 $$
 is an isomorphism of locally convex spaces. Consider the tables
  \begin{equation}\label{table_functions-1-1}
    \begin{gathered}
      \xymatrix@R=1.5em@C=0.6em{
        \cD \ar@{}[r]|-*[@]\txt{$\subseteq$} \ar@{}[d]|-*\txt{\rotatebox[origin=c]{-90}{$\cong$}} &
        \cD^F \ar@{}[r]|-*[@]\txt{$\subseteq$} \ar@{}[d]|-*\txt{\rotatebox[origin=c]{-90}{$\cong$}} &
        \cS \ar@{}[r]|-*[@]\txt{$\subseteq$} \ar@{}[d]|-*\txt{\rotatebox[origin=c]{-90}{$\cong$}} &
        \cD_{L^p} \ar@{}[r]|-*[@]\txt{$\subseteq$} \ar@{}[d]|-*\txt{\rotatebox[origin=c]{-90}{$\cong$}} &
        \dot\cB \ar@{}[r]|-*[@]\txt{$\subseteq$} \ar@{}[d]|-*\txt{\rotatebox[origin=c]{-90}{$\cong$}} &
        \cD_{L^\infty} \ar@{}[r]|-*[@]\txt{$\subseteq$} \ar@{}[d]|-*\txt{\rotatebox[origin=c]{-90}{$\cong$}} &
        \cO_C \ar@{}[r]|-*[@]\txt{$\subseteq$} \ar@{}[d]|-*\txt{\rotatebox[origin=c]{-90}{$\cong$}} &
        \cO_M \ar@{}[r]|-*[@]\txt{$\subseteq$} \ar@{}[d]|-*\txt{\rotatebox[origin=c]{-90}{$\cong$}} &
        \cE \ar@{}[d]|-*\txt{\rotatebox[origin=c]{-90}{$\cong$}}
        \\
        \C^{(\Z)} \widehat\otimes_\iota s \ar@{}[r]|-*[@]\txt{$\subseteq$} &
        \C^{(\Z)} \widehat\otimes s \ar@{}[r]|-*[@]\txt{$\subseteq$} &
        s \widehat\otimes s \ar@{}[r]|-*[@]\txt{$\subseteq$} &
        \ell^p \widehat\otimes s \ar@{}[r]|-*[@]\txt{$\subseteq$} &
        c_0 \widehat\otimes s \ar@{}[r]|-*[@]\txt{$\subseteq$} &
        \ell^\infty  \widehat\otimes s \ar@{}[r]|-*[@]\txt{$\subseteq$} &
        s' \widehat\otimes_\iota s \ar@{}[r]|-*[@]\txt{$\subseteq$} &
        s' \widehat\otimes s \ar@{}[r]|-*[@]\txt{$\subseteq$} &
        \C^{\Z} \widehat\otimes s
      }
    \end{gathered}
  \end{equation}
    \begin{equation}\label{table_distributions-1-1}
    \begin{gathered}
      \xymatrix@R=1.5em@C=1em{
        \cE' \ar@{}[r]|-*[@]\txt{$\subseteq$} \ar@{}[d]|-*\txt{\rotatebox[origin=c]{-90}{$\cong$}} &
        \cO_M' \ar@{}[r]|-*[@]\txt{$\subseteq$} \ar@{}[d]|-*\txt{\rotatebox[origin=c]{-90}{$\cong$}} &
        \cO_C' \ar@{}[r]|-*[@]\txt{$\subseteq$} \ar@{}[d]|-*\txt{\rotatebox[origin=c]{-90}{$\cong$}} &
        \cD'_{L^p} \ar@{}[r]|-*[@]\txt{$\subseteq$} \ar@{}[d]|-*\txt{\rotatebox[origin=c]{-90}{$\cong$}} &
        \dot\cB' \ar@{}[r]|-*[@]\txt{$\subseteq$} \ar@{}[d]|-*\txt{\rotatebox[origin=c]{-90}{$\cong$}} &
        \cD'_{L^\infty} \ar@{}[r]|-*[@]\txt{$\subseteq$} \ar@{}[d]|-*\txt{\rotatebox[origin=c]{-90}{$\cong$}} &
        \cS' \ar@{}[r]|-*[@]\txt{$\subseteq$} \ar@{}[d]|-*\txt{\rotatebox[origin=c]{-90}{$\cong$}} &
        \cD' \ar@{}[d]|-*\txt{\rotatebox[origin=c]{-90}{$\cong$}}
        \\
        \C^{(\Z)} \widehat\otimes s' \ar@{}[r]|-*[@]\txt{$\subseteq$} &
        s \widehat\otimes_\iota s' \ar@{}[r]|-*[@]\txt{$\subseteq$} &
        s \widehat\otimes s' \ar@{}[r]|-*[@]\txt{$\subseteq$} &
        \ell^p \widehat\otimes s'  \ar@{}[r]|-*[@]\txt{$\subseteq$} &
        c_0 \widehat\otimes s'  \ar@{}[r]|-*[@]\txt{$\subseteq$} &
        \ell^\infty \widehat\otimes s' \ar@{}[r]|-*[@]\txt{$\subseteq$} &
        s' \widehat\otimes s' \ar@{}[r]|-*[@]\txt{$\subseteq$} &
        \C^{\Z} \widehat\otimes s'
      }
    \end{gathered}
  \end{equation}
  where the test function and distribution spaces are defined over $\R$ and the sequence spaces are defined over the index set $\Z \times \N$. Then, the restriction of $\widetilde{C}_\psi$  to  each of the spaces in the upper row of \eqref{table_functions-1-1} and \eqref{table_distributions-1-1}  is an isomorphism onto the corresponding space in the lower row. 
 \end{theorem}
Theorem \ref{Wilson-intro} and a simple reenumeration procedure show the isomorphisms in~\eqref{table_functions} and~\eqref{table_distributions} in an explicit way and prove Theorem \ref{thm-intro}.

An explicit but somewhat complicated common unconditional basis for the separable spaces in~\eqref{table_functions} has been obtained in~\cite{BExplicit}. Its coefficient functionals form a common unconditional basis for the separable spaces in~\eqref{table_distributions}. We shall improve these results by showing that, if  $\psi \in \cD(\R)$ is such that $\mathcal{W}(\psi)$ is a Wilson basis for $L^2(\R)$,  $\mathcal{W}(\psi)$ is a common unconditional basis for all separable spaces in~\eqref{table_functions} and~\eqref{table_distributions} (see Corollary \ref{prop:CommonBasis} below).

 Finally, we would like to point out that we only consider  the one-dimensional case to avoid some technicalities. 
 By using the tensor product  approach explained in~\cite[Section~12.3]{FTFA}, it is possible to extend the results from this paper  to higher dimensions.

\section{Sequence spaces}\label{sect-sequence spaces}
Let $I = \N$ or $\Z$.  As customary, we define the Banach spaces
\[ \ell^p(I) = \Big\{ c= (c_i)_{i} \in \C^{I} \ |\  \| c \| = \Big ( \sum_{i \in I} |c_i|^p \Big)^{1/p} < \infty \Big\}, \qquad 1 \leq p < \infty, \]
\[ \ell^\infty(I) = \{ c= (c_i)_{i} \in \C^{I} \ |\  \| c \| = \sup_{i \in I} |c_i|< \infty \}. \]
 For  $l \in \Z$ we define 
\[ (c_0(I))_l = \{ c= (c_i)_{i} \in \C^{I}\ |\ \lim_{|i| \to \infty} (1+\abso{i}^2)^{l/2} c_i = 0\}; \]
endowed with the norm
\[
\| c \| = \sup_{i \in I} (1+\abso{i}^2)^{l/2} |c_i|
\]
it becomes a Banach space. We simply write $c_0(I) = (c_0(I))_0$.  We set
$$
s(I) = \varprojlim_{l \in \N} (c_0(I))_l, \qquad s'(I) = \varinjlim_{l \in \N} (c_0(I))_{-l}.
$$
We will drop the index set $I$ from the notation if it is clear from the context. 

Given two Hausdorff locally convex spaces $X$ and $Y$, we write $X \widehat{\otimes}_{\pi} Y$, $X \widehat{\otimes}_{\varepsilon} Y$, and $X \widehat{\otimes}_{\iota} Y$ for the completion of the tensor product $X \otimes Y$ with respect to the projective, injective, and inductive topology, respectively. We simply write   $X \widehat{\otimes} Y  = X \widehat{\otimes}_{\pi} Y = X \widehat{\otimes}_{\varepsilon} Y$ if either $X$ or $Y$ is nuclear; this convention was already  tacitly used in  \eqref{table_functions} and \eqref{table_distributions}. We refer to \cite{Grothendieck,Jarchow,Komatsu3}  for more information on completed tensor products.

Although the use of completed tensor products gives a succinct and elegant way of representing the sequence spaces occurring in \eqref{table_functions} and \eqref{table_distributions},  it is sometimes more convenient to view them as concrete double sequence spaces. We now give such representations. Let $I = \N$ or $\Z$ and let $X$ be a Banach space. For $1 \leq  p \leq \infty$ we define
\[ \ell^p\{I,X\} = \{ (x_i)_{i} \in X^{I} \ |\ (\|x_i\|_X)_{i} \in \ell^p(I) \}. \]
Similarly, for $l \in \Z$ we define 
\[ (c_0\{I;X\})_l = \{ (x_i)_{i} \in X^{I}\ |\ \lim_{|i| \to \infty} (1+\abso{i}^2)^{l/2} x_i = 0 \textrm{ in } X\}. \]
We endow these spaces  with their natural Banach space structure. As before, we simply write $c_0\{I;X\} = (c_0\{I;X\})_0$. Let $J= \N$ or $\Z$. We have the following canonical isomorphisms of locally convex spaces ($1 \leq p \leq \infty$):
\begin{center}
\begin{tabular}{ l|l} 
$\C^{(I)} \widehat\otimes_\iota s(J) \cong s(J)^{(I)}$ &     $\C^{(I)} \widehat\otimes s'(J) \cong s'(J)^{(I)}$ \\
$\C^{(I)} \widehat\otimes s(J) \cong \varprojlim_l  (c_0(J))_l^{(I)}$  &  $\ell^p(I) \widehat\otimes s'(J) \cong \varinjlim_l  (c_0\{J;\ell^p(I)\})_{-l}$  \\
$s(I) \widehat \otimes s(J) \cong \varprojlim_l\varprojlim_k (c_0\{I;(c_0(J))_{k}\})_l$ &  $c_0(I)\widehat\otimes s'(J) \cong \varinjlim_l  c_0\{I;(c_0(J))_{-l}\}$  \\ 
$\ell^p(I) \widehat\otimes s(J) \cong \varprojlim_l (c_0\{J;\ell^p(I)\})_l$ & $s'(I) \widehat \otimes s'(J) \cong \varinjlim_l \varinjlim_k  (c_0\{I;(c_0(J))_{-k}\})_{-l}$ \\ 
$c_0(I) \widehat \otimes s(J) \cong \varprojlim_l  c_0\{I;(c_0(J))_{l}\}$ & $\C^I \widehat\otimes s'(J) \cong s'(J)^I$  \\
$s'(I) \widehat\otimes_\iota s(J) \cong \varinjlim_l \varprojlim_k (c_0\{I;(c_0(J))_{k}\})_{-l}$ &  \\
 $s'(I) \widehat\otimes s(J) \cong  \varprojlim_k\varinjlim_l (c_0\{I;(c_0(J))_{k} \})_{-l}$ & \\
 $\C^I \widehat\otimes s(J) \cong s(J)^I$ & 
\end{tabular}
\end{center}
The isomorphism 
$$
s'(I) \widehat\otimes_\iota s(J) \cong \varinjlim_l \varprojlim_k (c_0\{I;(c_0(J))_{k}\})_{-l}
$$ 
is shown in \cite[p.~323]{BargetzOC}. All other isomorphisms follow directly from general results about completed tensor products and completed tensor product representations of vector-valued sequence spaces \cite{Grothendieck,Jarchow,Komatsu3}. 

We have the following duality relations, where $1 < p \leq \infty$ and $q$ denotes the conjugate index of $p$:
\begin{center}
\begin{tabular}{ l|l} 
$ \C^{(I)} \widehat\otimes s'(J) \cong (\C^{I} \widehat\otimes s(J))' $ & $\ell^p(I) \widehat\otimes s'(J) \cong  (\ell^q(I) \widehat{\otimes} s(J))'$  \\
$ s(I) \widehat\otimes_\iota s'(J)  \cong (s'(I) \widehat\otimes s(J))' $ &   $s'(I) \widehat\otimes s'(J)  \cong  (s(I) \widehat{\otimes} s(J))'$ \\
     $ s(I) \widehat\otimes s'(J) \cong (s'(I) \widehat\otimes_\iota s(J))'   $&  $\C^{I} \widehat\otimes s'(J) \cong (\C^{(I)} \widehat\otimes_\iota s(J))'$ \\
     $ \ell^1(I) \widehat\otimes s'(J) \cong (c_0(I) \widehat{\otimes} s(J))' $&
 \end{tabular}   
\end{center}
 These identifications hold in the topological sense if we endow the dual spaces with their strong topology. We recall that the space $s'(I) \widehat\otimes s(J)$ is bornological \cite[Chap.\ II, p.\ 128]{Grothendieck}. Hence, the isomorphism  $s(I) \widehat\otimes_\iota s'(J)  \cong (s'(I) \widehat\otimes s(J))'$ follows from  \cite[Chap.\ II, p.\ 90]{Grothendieck} (cf.\ the proof of \cite[Proposition 1]{BargetzOC}). Since  $s'(I) \widehat\otimes s(J)$ is reflexive (as it is bornological), we also obtain the isomorphism  $s(I) \widehat\otimes s'(J)  \cong (s'(I) \widehat\otimes_\iota s(J))'$. All other isomorphisms follow from general duality results about completed tensor products \cite{Grothendieck,Komatsu3}.
 
\section{Wilson bases of $\mathscr{D}'$}

The goal of this section is to  extend the notion of a Wilson basis to $\mathscr{D}'$. We start by introducing two auxiliary operators that will play an important role throughout this article (cf.\ \cite[Chapter 6]{FTFA}).  Fix lattice parameters $a,b > 0$. For $\psi \in \cD$ we define the \emph{analysis  operator} as
\begin{equation}
 C_{\psi,a,b}(f) = C_\psi (f) = (\langle f, \overline{T_{ak}M_{bn}\psi} \rangle )_{(k,n) \in \Z^{2}},  \qquad f \in \cD',
\label{analysis-def}
\end{equation}
and the  \emph{synthesis  operator} as
\begin{equation}
 D_{\psi,a,b}(c) =  D_\psi(c) = \sum_{(k,n) \in \Z^{2}} c_{k,n} T_{ak}M_{bn}\psi, \qquad c = (c_{k,n})_{k,n} \in \C^{\Z} \widehat\otimes s'(\Z).
\label{synthesis-def}
\end{equation}
We need the following lemma.
\begin{lemma}\label{frame-operators-D} Let $a,b > 0$ and let $\psi \in \cD$.
  \begin{itemize}
  \item[$(i)$] The mapping
    \[
      C_\psi = C_{\psi, \cD} : \cD \rightarrow  \C^{(\Z)} \widehat\otimes_\iota s(\Z)
    \]
    is well-defined and continuous. 
  \item[$(ii)$] For each $c \in \C^{(\Z)} \widehat\otimes_\iota s(\Z)$  the series 
  $$
  D_\psi(c) = \sum_{(k,n) \in \Z^{2}} c_{k,n} T_{ak}M_{bn}\psi
  $$
   is unconditionally convergent in  $\cD$. The mapping 
    \[
      D_\psi = D_{\psi, \cD} :   \C^{(\Z)} \widehat\otimes_\iota  s(\Z) \rightarrow \cD 
    \] 
    is well-defined and continuous. 
      \end{itemize}
\end{lemma}
\begin{proof}
  This follows from a straightforward computation (cf.\ Proposition \ref{frame-operators-functions} below). 
\end{proof}
Set
\[
  i_2(c) = (c_{k,-n})_{k,n}, \qquad c \in \C^{\Z} \widehat\otimes s'(\Z).
\]
Then, the mappings $i_2: \C^{\Z} \widehat\otimes s'(\Z) \rightarrow \C^{\Z} \widehat\otimes s'(\Z)$ and $i_2: \C^{(\Z)} \widehat\otimes_\iota s(\Z) \rightarrow \C^{(\Z)} \widehat\otimes_\iota s(\Z)$ are both isomorphisms and each other's transposes. 
\begin{lemma}\label{auxiliary-operators-dual} Let $a,b > 0$ and let $\psi \in \cD$.
  \begin{itemize}
  \item[$(i)$] The mapping
    \[
      C_\psi = C_{\psi, \cD'} : \cD' \rightarrow  \C^{\Z} \widehat\otimes s'(\Z)
    \]
    is well-defined and continuous. Moreover,
    \begin{equation}
      \label{transpose-1}
      C_{\psi, \cD'} = i_2 \circ {}^tD_{\overline{\psi}, \cD}.
    \end{equation}
  \item[$(ii)$] For each $c \in \C^{\Z} \widehat\otimes s'(\Z)$  the series 
  \begin{equation}
   D_\psi(c) = \sum_{(k,n) \in \Z^{2}} c_{k,n} T_{ak}M_{bn}\psi
  \label{series-D}
  \end{equation}
 is unconditionally convergent  in $\cD'$. The mapping
    \[
      D_\psi = D_{\psi, \cD'} :   \C^{\Z} \widehat\otimes s'(\Z) \rightarrow \cD'
    \] 
    is well-defined and continuous.  Moreover, 
    \begin{equation}
      \label{transpose-2}
      D_{\psi, \cD'} = {}^tC_{\overline{\psi}, \cD} \circ i_2.
    \end{equation}
    
  \end{itemize}
\end{lemma}
\begin{proof}
  $(i)$ Let $f \in \cD'$ be arbitrary. For all $c \in \C^{(\Z)} \widehat\otimes_\iota s(\Z)$ we have, by  Lemma \ref{frame-operators-D}$(ii)$,
  \begin{align*}
    \langle  i_2 \circ {}^tD_{\overline{\psi}, \cD}(f), c \rangle
    &= \langle f, D_{\overline{\psi}, \cD}(i_2(c)) \rangle = \langle f, \sum_{(k,n)  \in \Z^{2}} c_{k,n}  T_{ak}M_{-bn} \overline{\psi} \rangle \\
    &=  \sum_{(k,n) \in \Z^{2}} c_{k,n}  \langle f, T_{ak}M_{-bn} \overline{\psi} \rangle  = \langle C_{\psi,\cD'}(f),c \rangle.
  \end{align*}
 This shows that $C_{\psi, \cD'} = i_2 \circ {}^tD_{\overline{\psi}, \cD}$ and thus also that $C_{\psi, \cD'}$ is well-defined and continuous.

  $(ii)$  Let $c \in    \C^{\Z} \widehat\otimes s'(\Z)$ be arbitrary. Lemma \ref{frame-operators-D}$(i)$ implies that  the series  \eqref{series-D} is unconditionally convergent in $\mathscr{D}'$. Hence, we have that for all $\varphi \in \mathscr{D}$ 
  \[
      \langle  D_{\psi, \cD'}(c), \varphi \rangle =  \sum_{(k,n) \in \Z^{2}} c_{k,n} \langle T_{ak}M_{bn} \psi, \varphi \rangle =  \langle c, i_2( C_{\overline{\psi}, \cD}(\varphi)) \rangle = \langle {}^tC_{\overline{\psi}, \cD}(i_2(c)), \varphi \rangle.
  \]
  This shows that $D_{\psi, \cD'} = {}^tC_{\overline{\psi}, \cD} \circ i_2$ and thus also that $D_{\psi, \cD'}$ is continuous.
\end{proof}
Next, we consider Wilson systems on $\cD'$. For $\psi \in \cD$ we define the \emph{Wilson analysis operator} as
\[
\widetilde{C}_\psi (f) = (\langle f, \overline{\psi_{k,n}} \rangle )_{(k,n) \in \Z \times \N},  \qquad f \in \cD',
\]
and the  \emph{Wilson synthesis  operator} as
\[
 \widetilde{D}_\psi(c) = \sum_{(k,n) \in \Z \times \N} c_{k,n} \psi_{k,n}, \qquad c \in \C^{\Z} \widehat\otimes s'(\N),
\]
where the functions $\psi_{k,n}$ are defined in \eqref{eq:Wilson1D}. We now relate the above mappings to the analysis and synthesis operators  introduced in \eqref{analysis-def} and \eqref{synthesis-def}.  For $c \in \C^{\Z} \widehat\otimes s'(\Z)$ we define the element $v(c) \in  \C^{\Z} \widehat\otimes s'(\N)$ via
\begin{equation}
\label{def-v}
    v(c)_{k,n} =
    \begin{cases}
      c_{2k,0}, & n=0,\\
      \frac{1}{\sqrt{2}} (c_{k,n} + (-1)^{n+k}c_{k,-n}), & n>0.
    \end{cases}
\end{equation}
For $c \in \C^{\Z} \widehat\otimes s'(\N)$ we define the element $w(c) \in  \C^{\Z} \widehat\otimes s'(\Z)$ via
\begin{equation}
\label{def-w}
    w(c)_{k,n} =
    \begin{cases}
      0, & n=0\; \text{and}\; k\;\text{odd}, \\
      c_{\frac{k}{2},0}, & n=0\; \text{and}\; k\;\text{even}, \\
      \frac{1}{\sqrt{2}} c_{k,n} & n>0, \\
      \frac{(-1)^{k+n}}{\sqrt{2}} c_{k,-n} & n<0 .
    \end{cases}
\end{equation}
Note that the mappings
\[
  v\colon \C^{\Z} \widehat\otimes s'(\Z) \rightarrow \C^{\Z} \widehat\otimes s'(\N) \quad \mbox{and} \quad
  w\colon  \C^{\Z} \widehat\otimes s'(\N) \rightarrow \C^{\Z} \widehat\otimes s'(\Z)
\]
are well-defined and continuous.
\begin{proposition}\label{Wilson-operators-D-dual} Let $\psi \in \cD$. Then,
  \begin{itemize}
  \item[$(i)$] The mapping
    \[
      \widetilde{C}_{\psi} : \cD' \rightarrow  \C^{\Z} \widehat\otimes s'(\N)
    \]
    is well-defined and continuous. Moreover, 
    \begin{equation}
      \label{representation-1}
       \widetilde{C}_{\psi}  = v \circ  C_{\psi, \frac{1}{2},1}.
    \end{equation}
  \item[$(ii)$]  For each $c \in \C^{\Z} \widehat\otimes s'(\N)$  the series 
  $$
  \widetilde{D}_\psi(c) = \sum_{(k,n) \in \Z \times \N} c_{k,n} \psi_{k,n}
  $$ 
  is unconditionally convergent in $\cD'$. The mapping
    \[
      \widetilde{D}_\psi :   \C^{\Z} \widehat\otimes s'(\N) \rightarrow \cD'
    \] 
    is well-defined and continuous. Moreover, 
    \begin{equation}
      \label{representation-2}
      \widetilde{D}_{\psi}  = D_{\psi, \frac{1}{2},1} \circ w.
    \end{equation}
   
  \end{itemize}
\end{proposition}
\begin{proof}
The representations \eqref{representation-1} and \eqref{representation-2} follow from an easy computation. Hence, the result is a consequence of Lemma \ref{auxiliary-operators-dual} and the continuity of the mappings $v$ and $w$. 
\end{proof}
The following consequence of Proposition \ref{Wilson-operators-D-dual} is the foundation of our work.
\begin{theorem}\label{prop:WilsonBasisDprime}
 Let $\psi \in \cD$ be such that $\mathcal{W}(\psi)$ is a Wilson basis for $L^2$. Then,  $\mathcal{W}(\psi)$ is an unconditional basis of $\cD'$. The associated collection of coefficient functionals is given by the Wilson analysis operator  $\widetilde{C}_{\psi} : \cD' \rightarrow  \C^{\Z} \widehat\otimes s'(\N)$.
\end{theorem}

\section{Continuity of the analysis and synthesis operator}
Motivated by the representations \eqref{representation-1} and \eqref{representation-2}, we investigate in this section the mapping properties of  the analysis and synthesis operator on  the spaces occurring in \eqref{table_functions} and \eqref{table_distributions}. The \emph{short-time Fourier transform} (STFT) of $f \in \cD'$ with respect to a window  $\psi \in \cD$ is defined as
\[
  V_\psi f(x,\xi) :=  \langle f, \overline{M_\xi T_x\psi} \rangle, \qquad (x,\xi) \in \R^{2}.
\]
Obviously, $V_\psi f \in C(\R^{2})$. We refer to \cite[Section 2]{DV21} for a detailed study of the STFT on  $\cD'$. Note that, for $a,b > 0$ fixed,
\begin{equation}
|C_\psi(f)_{k,n}| = |V_\psi f(ak, bn)|, \qquad (k,n) \in \Z^2.
\label{linkSTFT}
\end{equation}
\begin{proposition} \label{frame-operators-functions} Let $a,b > 0$ and let $\psi \in \cD$. Consider the diagram
  \begin{equation}\label{table_functions-1}
    \begin{gathered}
      \xymatrix@R=1.5em@C=0.6em{
        \cD \ar@{}[r]|-*[@]\txt{$\subseteq$} \ar@{}[d]|-*\txt{\rotatebox[origin=c]{-90}{$\leftrightarrow$}} &
        \cD^F \ar@{}[r]|-*[@]\txt{$\subseteq$} \ar@{}[d]|-*\txt{\rotatebox[origin=c]{-90}{$\leftrightarrow$}} &
        \cS \ar@{}[r]|-*[@]\txt{$\subseteq$} \ar@{}[d]|-*\txt{\rotatebox[origin=c]{-90}{$\leftrightarrow$}} &
        \cD_{L^p} \ar@{}[r]|-*[@]\txt{$\subseteq$} \ar@{}[d]|-*\txt{\rotatebox[origin=c]{-90}{$\leftrightarrow$}} &
        \dot\cB \ar@{}[r]|-*[@]\txt{$\subseteq$} \ar@{}[d]|-*\txt{\rotatebox[origin=c]{-90}{$\leftrightarrow$}} &
        \cD_{L^\infty} \ar@{}[r]|-*[@]\txt{$\subseteq$} \ar@{}[d]|-*\txt{\rotatebox[origin=c]{-90}{$\leftrightarrow$}} &
        \cO_C \ar@{}[r]|-*[@]\txt{$\subseteq$} \ar@{}[d]|-*\txt{\rotatebox[origin=c]{-90}{$\leftrightarrow$}} &
        \cO_M \ar@{}[r]|-*[@]\txt{$\subseteq$} \ar@{}[d]|-*\txt{\rotatebox[origin=c]{-90}{$\leftrightarrow$}} &
        \cE \ar@{}[d]|-*\txt{\rotatebox[origin=c]{-90}{$\leftrightarrow$}}
        \\
        \C^{(\Z)} \widehat\otimes_\iota s \ar@{}[r]|-*[@]\txt{$\subseteq$} &
        \C^{(\Z)} \widehat\otimes s \ar@{}[r]|-*[@]\txt{$\subseteq$} &
        s \widehat\otimes s \ar@{}[r]|-*[@]\txt{$\subseteq$} &
        \ell^p \widehat\otimes s \ar@{}[r]|-*[@]\txt{$\subseteq$} &
        c_0 \widehat\otimes s \ar@{}[r]|-*[@]\txt{$\subseteq$} &
        \ell^\infty  \widehat\otimes s \ar@{}[r]|-*[@]\txt{$\subseteq$} &
        s' \widehat\otimes_\iota s \ar@{}[r]|-*[@]\txt{$\subseteq$} &
        s' \widehat\otimes s \ar@{}[r]|-*[@]\txt{$\subseteq$} &
        \C^{\Z} \widehat\otimes s
      }
    \end{gathered}
  \end{equation}
  where the sequence spaces are defined over the index set $\Z^2$. Then,
  \begin{enumerate}[label = (\roman*)]
  \item The analysis operator $C_\psi$ maps each space in the upper row of \eqref{table_functions-1} continuously into the corresponding space in the lower row.
  \item The synthesis operator $D_\psi$ maps each space in the lower row of  \eqref{table_functions-1} continuously into the corresponding space in the upper row.
  \end{enumerate}
\end{proposition}
\begin{proof} Set $K = \operatorname{supp} \psi$ and choose $R > 0$ such that $K \subseteq [-R,R]$.
  \begin{enumerate}[label = (\roman*), wide=0pt]   
\item For every $j \in  \N$ we have
  \begin{equation}
    \label{basic-ineq}
    |\xi^j V_\psi \varphi(x,\xi)| \leq (2\pi)^{-j} \|\psi\|_{\infty, j} \sum_{l \leq j} \binom{j}{l} \int_K |\varphi^{(l)}(x+t)| \dt, \qquad \varphi \in \mathscr{E}^{j}.
  \end{equation}
  Consequently, we have that for all $m \in \N$
  \begin{equation}    \label{basic-ineq-1}
    \begin{aligned}
    |V_\psi \varphi(x,\xi)|(1+|\xi|)^m  
    &\leq 2^m \max_{j \leq m} |\xi^j V_\psi \varphi(x,\xi)| \\ 
    &\leq 2^m |K| \|\psi\|_{\infty,m} \max_{j \leq m} \max_{t \in K} |\varphi^{(j)}(x+t)|, \qquad \varphi \in \mathscr{E}^{m}.
  \end{aligned}
  \end{equation}
  The above inequality and \eqref{linkSTFT} imply the result for $\cD$, $\cD^F$, $\dot{\cB}$, $\cD_{L^\infty}$, and $\cE$. By using the inequality
  \begin{equation}
    \label{subadd}
    (1+|x|)^l \leq (1+|t|)^{|l|}(1+|x+t|)^l, \qquad l \in \Z,
  \end{equation}
  we obtain from \eqref{basic-ineq-1} that for all $m \in \N$, $l \in \Z$
  $$
  |V_\psi \varphi(x,\xi)|(1+|\xi|)^m(1+|x|)^l \leq C_{m,l}  \max_{j \leq m} \sup_{t \in \R^d} |\varphi^{(j)}(t)|(1+|t|)^l, \qquad \varphi \in \mathscr{E}^{m},
  $$
  where $C_{m,l} = 2^m(1+R)^{|l|} |K|\|\psi\|_{\infty,m}$. The above inequality and \eqref{linkSTFT}  yield the result for $\cS$, $\cO_C$, and $\cO_M$. Finally, we consider  $\cD_{L^p}$ $(1 \leq p < \infty)$. By \cite[Theorem 12.2.3]{FTFA}, it suffices to show that for all $m \in \N$ the space $\cD_{L^p}$ is continuously included in the modulation space $M^{p,\infty}_{1 \otimes (1+|\, \cdot \,|)^m}$, or, equivalently, that the mapping $V_\psi : \cD_{L^p} \rightarrow L^{p,\infty}_{1 \otimes (1+|\, \cdot \,|)^m}$  is well-defined and continuous  (see \cite[Definition 11.3.1]{FTFA} for the definition of $M^{p,\infty}_{1 \otimes (1+|\, \cdot \,|)^m}$). The inequality  \eqref{basic-ineq} and H\"older's inequality imply that for all $\varphi \in \cD_{L^p}$
  \begin{align*}
    \| V_\psi \varphi \|_{L^{p,\infty}_{1 \otimes (1+|\, \cdot \,|)^m}} &\leq \sup_{\xi \in \R} \| V_\psi \varphi (\, \cdot \, , \xi) \|_{L^p}(1+|\xi|)^m \\
                                                                        &\leq 2^m \max_{j \leq m} \sup_{\xi \in \R^d} \| \xi^j V_\psi \varphi (\, \cdot \, , \xi) \|_{L^p} \\
                                                                        &\leq 2^m  \|\psi\|_{\infty, m} \max_{j \leq m} (2\pi)^{-j} \sum_{l \leq j} \binom{j}{l} \left ( \int_{\R^d} \left(\int_K |\varphi^{(l)}(x+t)| \dt\right)^p \dx \right)^{1/p} \\
                                                                        &\leq 2^m |K|^{(p-1)/p}  \|\psi\|_{\infty, m} \max_{l \leq m}   \left ( \int_{\R^d} \int_K |\varphi^{(l)}(x+t)|^p \dt \dx \right)^{1/p} \\
                                                                        &= 2^m |K|  \|\psi\|_{\infty, m} \| \varphi\|_{p,m}.
  \end{align*}
  This completes the proof of part $(i)$.

  $(ii)$
  For every $j \in \N$ we have
  \begin{equation}
    \label{basic-ineq-2}
    |(T_x M_\xi \psi)^{(j)}(t)| \leq (2\pi)^{j} (1+ |\xi|)^{j} \sum_{l \leq j} \binom{j}{l} |\psi^{(l)}(t-x)|.
  \end{equation}
  Consequently, we have that for all $j \in \N$
  \begin{equation}
    \label{basic-ineq-3}
    |(T_x M_\xi  \psi)^{(j)}(t)| \leq (4\pi)^{j} (1+ |\xi|)^{j} \max_{l \leq j} |\psi^{(l)}(t-x)|.
  \end{equation}
  The above inequality  implies the result for $\cD$, $\cD^F$, and $\cE$.  By using \eqref{subadd}, we obtain from \eqref{basic-ineq-3} that for all $j \in \N$, $l \in \Z$
  \[
    |(T_xM_\xi \psi)^{(j)}(t)|(1+|t|)^l \leq (4\pi)^{j}(1+R)^{|l|}  \|\psi\|_{\infty, j} (1+ |\xi|)^{j} (1+|x|)^l.
  \]
  The above inequality  implies the result for $\cS$, $\cO_C$, and $\cO_M$.  Next, we consider the spaces $\dot{\cB}$ and $\cD_{L^\infty}$. Note that
  \[
    D_\psi(c) = \sum_{k \in \Z}  \sum_{n \in \Z}  c_{k,n} T_{ak} M_{bn} \psi \mbox{ in $\cD'$,} \qquad c \in \C^{\Z} \widehat{\otimes} s'.
  \]
  Let $c \in \ell^\infty \widehat\otimes s$ be arbitrary. Inequality \eqref{basic-ineq-3} implies that for all $k \in \Z$ the series 
  \[
    \psi_k(c) = \sum_{n \in \Z}  c_{k,n} T_{ak} M_{bn} \psi 
  \]
  is unconditionally convergent in $\cD_{L^\infty}$ with
  \begin{equation}
    \label{norm-parts}
    \| \psi_k(c) \|_{\infty, m} \leq C_m \sup_{n \in \Z} |c_{k,n}|(1+|n|)^{m+2}, \qquad m \in \N,
  \end{equation}
  where $C_m = (4\pi)^m \max\{1,b^m\} \|\psi\|_{\infty,m} \sum_{n \in \Z} (1+|n|)^{-2}$. Moreover, $\operatorname{supp} \psi_k(c) \subseteq ak + K$. For $t \in \R$ we set $I_t = \{ k \in \Z \, | \, t \in ak + K \}$. Note that there is $B > 0$ such that $|I_t| \leq B$ for all $t \in \R$. Hence, by \eqref{norm-parts}, we have that for all $j \in \N$
\[
    |D^{(j)}_\psi(c)(t)| = \sum_{k \in I_t} |\psi_k^{(j)}(c)(t)| \leq BC_{j} \sup_{k \in I_t} \sup_{n \in \Z} |c_{k,n}|(1+|n|)^{j+2}.
\]
  The above inequality implies the result for $\cD_{L^\infty}$. Furthermore, this inequality also shows that $D_\psi(c) \in \dot{\cB}$ for all $c \in c_0 \widehat\otimes s$, which yields the result for  $\dot{\cB}$. Finally, we consider $\cD_{L^p}$ $(1 \leq p < \infty)$. We use the same idea as in the proof of \cite[Theorem 12.2.4]{FTFA}. Given a set $A \subseteq \R$, we write $1_A$ for the characteristic function of $A$. Since $K$ is compact, there is $e = (e_k)_{k \in \Z} \in \C^{(\Z)}$ (with only $0$ and $1$ as entries) such that
  $$
  1_K(t) \leq \sum_{k \in \Z} e_k T_{ak}1_{[0,a)}(t).
  $$
  Let $c \in \ell^p \widehat\otimes s$ be arbitrary. For all $j \in \N$, using the same notation as above, we have
  \begin{align*}
    |D^{(j)}_\psi(c)(t)| &\leq \sum_{k \in \Z} |\psi_k^{(j)}(c)(t)| \\
                              &\leq  \sum_{k \in \Z} \|\psi_k^{(j)}(c)\|_\infty T_{ak} 1_K(t) \\
                              &\leq \sum_{k \in \Z} \sum_{l \in \Z} \|\psi_k^{(j)}(c)\|_\infty e_{l} T_{a(k+l)} 1_{[0,a)}(t)\\
                              &\leq \sum_{k \in \Z} \sum_{l \in \Z} \|\psi_k^{(j)}(c)\|_\infty e_{l-k} T_{al} 1_{[0,a)}(t)\\
                              &= \sum_{l \in \Z} ((\|\psi_k^{(j)}(c)\|_\infty)_{k \in \Z} \ast e) (l) T_{al} 1_{[0,a)}(t).
  \end{align*}
  Hence, by \eqref{norm-parts} and Young's convolution inequality for $\ell^p$, we obtain that for all $m \in \N$
  \begin{align*}
    \| D_\psi(c)\|_{p,m} &\leq \left ( \int_{\R} \left | \sum_{l \in \Z} ((\|\psi_k(c)\|_{m,\infty})_{k \in \Z} \ast e) (l) T_{al} 1_{[0,a)}(t) \right|^p \dt \right)^{1/p} \\
                         &= \left ( \int_{\R} \sum_{l \in \Z} ((\|\psi_k(c)\|_{m,\infty})_{k \in \Z} \ast e)^p (l) T_{al} 1_{[0,a)}(t) \dt \right)^{1/p} \\
                         &= a^{1/p} \| (\|\psi_k(c)\|_{m,\infty})_{k \in \Z} \ast e \|_{\ell^p} \\
                         &\leq a^{1/p} \| (\|\psi_k(c)\|_{m,\infty})_{k \in \Z} \|_{\ell^p} \| e \|_{\ell^1} \\
                         &\leq  C_m a^{1/p} \| e \|_{\ell^1}  \| (\sup_{n \in \Z} |c_{k,n}|(1+|n|)^{m+2})_{k \in \Z} \|_{\ell^p} \\
                         &\leq C'_m \sup_{n \in \Z}  \| (c_{k,n})_{k \in \Z} \|_{\ell^p} (1+|n|)^{m+2 + (2/p)},
  \end{align*}
where $C'_m = C_m a^{1/p} \| e \|_{\ell^1}\left( \sum_{n \in \Z} (1+|n|)^{-2}\right)^{1/p}$. This completes the proof of part $(ii)$.
\end{enumerate}
\end{proof}

\begin{proposition}\label{frame-operators-distributions} Let $a,b > 0$ and let $\psi \in \cD$. Consider the diagram
  \begin{equation}\label{table_distributions-1}
    \begin{gathered}
      \xymatrix@R=1.5em@C=1em{
        \cE' \ar@{}[r]|-*[@]\txt{$\subseteq$} \ar@{}[d]|-*\txt{\rotatebox[origin=c]{-90}{$\leftrightarrow$}} &
        \cO_M' \ar@{}[r]|-*[@]\txt{$\subseteq$} \ar@{}[d]|-*\txt{\rotatebox[origin=c]{-90}{$\leftrightarrow$}} &
        \cO_C' \ar@{}[r]|-*[@]\txt{$\subseteq$} \ar@{}[d]|-*\txt{\rotatebox[origin=c]{-90}{$\leftrightarrow$}} &
        \cD'_{L^p} \ar@{}[r]|-*[@]\txt{$\subseteq$} \ar@{}[d]|-*\txt{\rotatebox[origin=c]{-90}{$\leftrightarrow$}} &
        \dot\cB' \ar@{}[r]|-*[@]\txt{$\subseteq$} \ar@{}[d]|-*\txt{\rotatebox[origin=c]{-90}{$\leftrightarrow$}} &
        \cD'_{L^\infty} \ar@{}[r]|-*[@]\txt{$\subseteq$} \ar@{}[d]|-*\txt{\rotatebox[origin=c]{-90}{$\leftrightarrow$}} &
        \cS' \ar@{}[r]|-*[@]\txt{$\subseteq$} \ar@{}[d]|-*\txt{\rotatebox[origin=c]{-90}{$\leftrightarrow$}} &
        \cD' \ar@{}[d]|-*\txt{\rotatebox[origin=c]{-90}{$\leftrightarrow$}}
        \\
        \C^{(\Z)} \widehat\otimes s' \ar@{}[r]|-*[@]\txt{$\subseteq$} &
        s \widehat\otimes_\iota s' \ar@{}[r]|-*[@]\txt{$\subseteq$} &
        s \widehat\otimes s' \ar@{}[r]|-*[@]\txt{$\subseteq$} &
        \ell^p \widehat\otimes s'  \ar@{}[r]|-*[@]\txt{$\subseteq$} &
        c_0 \widehat\otimes s'  \ar@{}[r]|-*[@]\txt{$\subseteq$} &
        \ell^\infty \widehat\otimes s' \ar@{}[r]|-*[@]\txt{$\subseteq$} &
        s' \widehat\otimes s' \ar@{}[r]|-*[@]\txt{$\subseteq$} &
        \C^{\Z} \widehat\otimes s'
      }
    \end{gathered}
  \end{equation}
    where the sequence spaces are defined over the index set $\Z^2$. Then,
  \begin{itemize}
  \item[$(i)$] The analysis operator $C_\psi$ maps each space in the upper row of \eqref{table_distributions-1} continuously into the corresponding space in the lower row.
  \item[$(ii)$] The synthesis operator $D_\psi$ maps each space in the lower row of \eqref{table_distributions-1} continuously into the corresponding space in the upper row. \end{itemize}
\end{proposition}
\begin{proof}
  For each space $X$ in the lower row of diagram \eqref{table_distributions-1} the mapping $i_2$ is an isomorphism from $X$ onto itself. Hence, for all spaces in the upper row of diagram \eqref{table_distributions-1}, besides $\dot{\cB}'$, $(i)$ follows from  \eqref{transpose-1} and Proposition \ref{frame-operators-functions}$(ii)$, while $(ii)$ follows from \eqref{transpose-2} and Proposition \ref{frame-operators-functions}$(i)$. We now show $(i)$ and $(ii)$ for $\dot{\cB}'$. We have that
  $c_0 \widehat\otimes s' = \overline{\C^{(\Z)} \widehat\otimes s'}^{{\ell}^\infty \widehat\otimes s'}$ and, by definition,  $\dot{\cB}' = \overline{\cE'}^{\cD'_{L^\infty}}$. Hence, $(i)$ follows from the continuity of the mapping $C_\psi: \cD'_{L^\infty} \rightarrow {\ell}^\infty \widehat\otimes s'$ and the inclusion $C_\psi(\mathscr{E}') \subseteq \C^{(\Z)} \widehat\otimes s'$, while $(ii)$ follows from the continuity of the mapping $D_\psi:  {\ell}^\infty \widehat\otimes s' \rightarrow \cD'_{L^\infty}$ and the inclusion $D_\psi(\C^{(\Z)} \widehat\otimes s') \subseteq \mathscr{E}'$.
\end{proof}

\section{Sequence space representations and bases via Wilson bases}
We are ready to provide sequence space representations and common unconditional bases for the test function and distribution spaces occurring in \eqref{table_functions} and \eqref{table_distributions}  via Wilson bases on $\cD'$. 
\begin{theorem} \label{USSR} Let $\psi \in \cD$ be such that $\mathcal{W}(\psi)$ is a Wilson basis for $L^2$. Then, the restriction of $\widetilde{C}_\psi$  to  each of the spaces in the upper row of \eqref{table_functions-1-1} and \eqref{table_distributions-1-1}  is an isomorphism onto the corresponding space in the lower row and its inverse is given by the restriction of $\widetilde{D}_\psi$ to this space.  
\end{theorem}
\begin{proof} The mapping  $v: \C^{\Z} \widehat\otimes s'(\Z) \rightarrow \C^{\Z} \widehat\otimes s'(\N)$  defined in  \eqref{def-v} maps each space in the lower row of  \eqref{table_functions-1} and \eqref{table_distributions-1}  continuously into its corresponding space  in the lower row of \eqref{table_functions-1-1} and \eqref{table_distributions-1-1}. Likewise,  the mapping  $w: \C^{\Z} \widehat\otimes s'(\N) \rightarrow \C^{\Z} \widehat\otimes s'(\Z)$  defined in  \eqref{def-w} maps each space in the lower row of \eqref{table_functions-1-1} and \eqref{table_distributions-1-1}  continuously into its corresponding space  in the lower row of  \eqref{table_functions-1} and \eqref{table_distributions-1}.  Hence, in view of the representations \eqref{representation-1} and \eqref{representation-2}, the result follows from Theorem \ref{prop:WilsonBasisDprime}, Proposition \ref{frame-operators-functions}, and Proposition \ref{frame-operators-distributions}.
 \end{proof}

\begin{corollary}\label{prop:CommonBasis}
Let $\psi \in \cD$ be such that $\mathcal{W}(\psi)$ is a Wilson basis for $L^2$. For every space $X$ in the upper row of \eqref{table_functions-1-1} and \eqref{table_distributions-1-1},  besides $\cD_{L^\infty}$ and $\cD'_{L^\infty}$, the system $\mathcal{W}(\psi)$  is an unconditional basis for $X$.
\end{corollary}
\begin{proof}
 Given a set $I \subseteq \Z \times \N$, we write $1_I$ for the characteristic sequence of $I$. For every space $Y$ in the lower row of  \eqref{table_functions-1-1} and \eqref{table_distributions-1-1}, besides ${\ell}^\infty \widehat\otimes {s}$ and ${\ell}^\infty \widehat\otimes {s}'$, the tensor product of the standard bases is an unconditional basis for $Y$. Hence, for all $c \in Y$ the net $\{c \cdot 1_I \, | \, I \subset \Z \times \N \mbox{ finite} \}$ converges to $c$ in $Y$.  Since  for all $I \subset \Z \times \N$ finite
  \[
    \widetilde{D}_\psi(c) - \sum_{(k,n) \in I} c_{k,n} \psi_{k,n} = \widetilde{D}_\psi(c -c \cdot 1_I)\mbox{ in $\cD'$,} \qquad c \in \C^{\Z} \widehat{\otimes} {s}'(\N),
  \]
the result follows from Theorem \ref{USSR}.
\end{proof}

\noindent {\bf Acknowledgements.} A.~Debrouwere was supported by FWO-Vlaanderen through the postdoctoral grant 12T0519N.

\printbibliography

\end{document}